\documentclass[12pt,a4]{article}

\usepackage{latexsym}
\usepackage{amsmath}
\usepackage{amssymb}
\usepackage{amsthm}
\usepackage{amscd}

\newtheorem{theorem}{Theorem}[section]
\newtheorem{lemma}[theorem]{Lemma}
\newtheorem{claim}[theorem]{Claim}

\newtheorem{proposition}[theorem]{Proposition}

\theoremstyle{definition}

\newtheorem{remark}[theorem]{Remark}

\theoremstyle{remark}

\def\al{\alpha}
\def\bt{\beta}

\def\dl{\delta}
\def\ve{\varepsilon}

\def\la{\lambda}

\def\cC{{\cal C}}
\def\cM{{\cal M}}
\def\cQ{{\cal Q}}
\def\cQmu{{\cal Q}(\mu)}

\def\cMpq{\cM^p_q}
\def\cCpq{\cC^p_q}

\newcommand{\N}{{\bf N}}

\newcommand{\R}{{\bf R}}

\def\l{\left}
\def\r{\right}
\def\ds{\displaystyle}
\def\scr{\scriptstyle}
\def\eqref#1{{\rm(\ref{#1})}}

\newcommand{\lft}[1]{\lefteqn{#1}}
\newcommand{\supp}{\,{\rm supp}\,}

\begin{document}

\title{Equivalent norms for the Morrey spaces\\ 
with non-doubling measures}

\maketitle

\begin{center}
\author{
{\it Yoshihiro Sawano, Hitoshi Tanaka}\footnote{
The first author is supported by Research Fellowships
of the Japan Society for the Promotion
of Science for Young Scientists.
The second author is supported by the 21st century COE program 
at Graduate School of Mathematical Sciences, 
the University of Tokyo 
and by F\=ujyukai foundation.
} \\
{\it Graduate School of Mathematical Sciences, 
The University of Tokyo, \\ 
3-8-1 Komaba, Meguro-ku
Tokyo 153-8914, JAPAN
} \\
E-mail: yosihiro@ms.u-tokyo.ac.jp, htanaka@ms.u-tokyo.ac.jp
}
\end{center}

\begin{abstract}
In this paper under some growth condition
we investigate the connection 
between RBMO and the Morrey spaces.
We do not assume
the doubling condition
which has been a key property
of harmonic analysis.
We also obtain another type of equivalent norms.
\end{abstract}

\noindent{\bf KEYWORDS :}\ \
Morrey space, \ Campanato space, \ equivalent norms

\noindent{\bf AMS Subject Classification :}\ \
Primary 42B35, Secondary 46E35.

\section{Introduction}
\noindent

In this paper we discuss 
equivalent norms for the (vector-valued) Morrey spaces 
with non-doubling measures.
We consider the connection 
between the Morrey spaces and the Campanato spaces 
with underlying measure $\mu$ non-doubling.
The Morrey spaces appeared  in \cite{Mo} originally 
in connection with the partial differential equations and 
the Campanato spaces in \cite{Ca1} and \cite{Ca2}.
We refer to \cite{Na} for the result 
of Morrey spaces coming with the doubling measures.
Before we state our main theorem,
let us make a brief view of the terminology
of measures on $\R^d$.
We say that 
a (positive) Radon measure $\mu$ on $\R^d$ 
satisfies the growth condition if 
\begin{equation} \label{growth}
\mu(Q(x,l)) \le C_0\,l^n
\mbox{ for all $x \in \supp(\mu)$ and $l>0$},
\end{equation}
where $C_0$ and $n \in (0,d]$ are some fixed numbers. 
A measure $\mu$ is said to satisfy the doubling condition if 
\[
\mu(Q(x,2l)) \le C\,\mu(Q(x,l))
\mbox{ for all $x \in \R^d$ and $l>0$}
\]
for some constant $C>0$. 
A measure $\mu$ 
which satisfies the growth condition 
will be called growth measure 
while 
a measure $\mu$ 
with the doubling condition 
will be called the doubling measure. 

By a ^^ ^^ cube" $Q \subset \R^d$ 
we mean a closed cube 
having sides parallel to the axes. 
Its center will be denoted by $z_Q$ and 
its side length by $\ell(Q)$. 
By $Q(x,l)$ we will also denote the cube 
centered at $x$ of sidelength $l$. 
For $\rho>0$, 
$\rho\,Q$ means a cube concentric to $Q$ with 
its sidelength $\rho\,\ell(Q)$. 
Let $\cQ(\mu)$ denote the set of all cubes $Q \subset \R^d$ 
with positive $\mu$-measures. 
If $\mu$ is finite, 
we include $\R^d$ in $\cQ(\mu)$ as well. 
In \cite{SaTa1}, the authors defined 
the Morrey spaces $\cMpq(k,\mu)$ for non-doubling measures 
normed by 
\[
\| f \, : \, \cMpq(k,\mu)\|
:=
\sup_{Q \in \cQmu}
\mu(k\,Q)^{\frac{1}{p}-\frac{1}{q}}
\l(\int_Q|f|^q\,d\mu\r)^{\frac{1}{q}},
\, 1 \le q \le p <\infty,
\, k>1.
\]
The fundamental property of this norm is
\[
\|f \, : \, \cMpq(k_1,\mu)\|
\le 
\|f \, : \, \cMpq(k_2,\mu)\|
\le C_d\l(\frac{k_1-1}{k_2-1}\r)^d
\|f \, : \, \cMpq(k_1,\mu)\|
\]
for $1<k_1<k_2<\infty$.
With this relation in mind, 
we will denote $\cMpq(\mu)=\cMpq(2,\mu)$.
The aim of this paper is 
to find some norms equivalent to this Morrey norm. 

\section{Equivalent norm of doubling type}
\noindent

In this section 
we investigate an equivalent norm related to the doubling cubes.
Although we now envisage the non-homogeneous setting,
we are still able to place ourselves 
in the setting of the doubling cubes.
In \cite{To2}, 
Tolsa defined the notion of doubling cubes. 
Let $k,\bt>1$.
We say that $Q \in \cQmu$ is a $(k,\bt)$-doubling cube, if 
$\mu(kQ) \le \bt\,\mu(Q)$.
It is well-known that, if $\bt>k^d$,
then for $\mu$-almost all $x \in \R^d$ 
and for all $Q \in \cQmu$ centered at $x$,
we can find a $(k,\bt)$-doubling cube
from $k^{-1}Q,k^{-2}Q,\ldots$.
In what follows
we denote by $\cQ(\mu;k,\bt)$
the set of all $(k,\bt)$-doubling cubes in $\cQmu$.
We fix $k,\bt>1$ with $\bt>k^d$.
Let $1 \le q \le p<\infty$. 
For $f \in L^1_{loc}(\mu)$ define 
\[
\|f \, : \, \cMpq(\mu)\|_d
:=
\sup_{Q \in \cQ(\mu;k,\bt)}
\mu(Q)^{\frac{1}{p}-\frac{1}{q}}
\l(\int_Q |f(y)|^q\,d\mu(y)\r)^\frac{1}{q}.
\]
Now we present the main theorem in this section.
\begin{theorem} \label{thm1}
Let $\mu$ be a Radon measure 
which does not necessarily satisfy 
the growth condition nor the doubling condition 
and let $1 \le q<p<\infty$. 
If $\ds\bt>k^{\frac{dpq}{p-q}}$, 
then 
\[
C^{-1}\,\| f \, : \, \cMpq(\mu)\|_d 
\le \| f \, : \, \cMpq(\mu)\| \le
C\,\| f \, : \, \cMpq(\mu)\|_d, 
\,f \in \cMpq(\mu),
\]
for some constant $C>0$.
\end{theorem}

Before we come to the proof of Theorem \ref{thm1},
two clarifying remarks may be in order.

\begin{remark} 
If $p=q$, this theorem fails in general.
However, 
if we assume the growth condition or the doubling condition, 
the theorem is still available for $p=q$.
In fact, 
under the growth condition or the doubling condition 
for any cube $Q \in \cQmu$ 
we can find a large integer $j \gg 1$
such that $2^jQ \in \cQ(\mu;k,\bt)$.
\end{remark}

\begin{remark} 
This theorem readily extends to the vector-valued version.
Let $1 \le q \le p<\infty$ and $r \in (1,\infty)$. 
We define the vector-valued Morrey spaces $\cMpq(l^r,\mu)$ 
by the set of sequences of $\mu$-measurable functions $\{f_j\}_{j \in \N}$
for which 
\[
\|f_j \, : \, \cMpq(l^r,\mu)\|
:=
\sup_{Q \in \cQ(\mu)}
\mu(2Q)^{\frac{1}{p}-\frac{1}{q}}
\l(\int_Q\|f_j \, : \, l^r\|^q\,d\mu\r)^\frac{1}{q}
<\infty.
\]
The theorem can be extended to the vector valued version. Let 
\[
\|f_j \, : \, \cMpq(l^r,\mu)\|_d
:=
\sup_{Q \in \cQ(\mu;k,\bt)}
\mu(Q)^{\frac{1}{p}-\frac{1}{q}}
\l(\int_Q \| f_j(y) \, : \, l^r \|^q\,d\mu(y)\r)^{\frac{1}{q}}.
\]
Then
$\ds
C^{-1}\,\| f_j \, : \, \cMpq(l^r,\mu)\|_d
\le
\| f_j \, : \, \cMpq(l^r,\mu)\|
\le
C\,\| f_j \, : \, \cMpq(l^r,\mu)\|_d.
$
The same proof as the scalar-valued spaces 
works for the vector-valued spaces,
so in the actual proof 
we concentrate on the scalar-valued cases.
\end{remark}

\begin{proof}
Given $k>1$, we shall prove 
\[
C^{-1}\,\| f \, : \, \cMpq(\mu)\|_d
\le \| f \, : \, \cMpq(k,\mu)\|,
\quad
\| f \, : \, \cMpq(\mu)\|
\le C\,\| f \, : \, \cMpq(\mu)\|_d
\]
for large $\bt>0$.
The left inequality is obvious, so let us prove the right inequality.
We have only to show that, 
for every cube $Q \in \cQmu$,
\[
\mu(2Q)^{\frac{1}{p}-\frac{1}{q}}
\l(\int_Q |f(y)|^q\,d\mu(y)\r)^\frac{1}{q}
\le C\,\| f \, : \, \cMpq(\mu)\|_d.
\]

Let $x \in Q \cap \supp(\mu)$ and 
$Q(x)$ the largest doubling cube 
centered at $x$ and having sidelength $k^{-j}\ell(Q)$
for some $j \in \N$.
Existence of $Q(x)$ can be ensured 
for $\mu$-almost all $x \in \R^d$.
Set 
\[
\cQ_0(j)
:=
\{Q(x) \, : \, \ell(Q(x))=k^{-j}\ell(Q)\},
\, j \in \N.
\]
By Besicovitch's covering lemma 
we can take $\cQ(j) \subset \cQ_0(j)$ 
so that
$\ds\sum_{R \in \cQ(j)}\chi_R \le 4^d\chi_{2Q}$
and that $\ds x \in \bigcup_{R \in \cQ(j)}R$
for $\mu$-almost all $x \in Q$ with 
$\ell(Q(x))=k^{-j}\ell(Q)$.
Volume argument gives us that 
$\sharp(\cQ(j)) \le 8^d\,k^{jd}$.
Since
$\ds
\l(\int_Q |f(y)|^q\,d\mu(y)\r)^{\frac{1}{q}}
\le \sum_{j=1}^{\infty} \sum_{R \in \cQ(j)}
\l(\int_R |f(y)|^q\,d\mu(y)\r)^{\frac{1}{q}}
$
and $\mu(R) \le \bt^{-j}\mu(2Q)$ 
for all $R \in \cQ(j)$, we have
\begin{eqnarray*}
\lft{
\mu(2Q)^{\frac{1}{p}-\frac{1}{q}}
\l(\int_Q |f(y)|^q\,d\mu(y)\r)^\frac{1}{q}
\le 
\sum_{j=1}^{\infty} 
\bt^{j\l(\frac{1}{p}-\frac{1}{q}\r)}
\sum_{R \in \cQ(j)}
\mu(R)^{\frac{1}{p}-\frac{1}{q}}
\l(\int_R |f(y)|^q\,d\mu(y)\r)^\frac{1}{q}
}\\ &\le& 
\sum_{j=1}^{\infty} 
8^d\,k^{jd}\,\bt^{j\l(\frac{1}{p}-\frac{1}{q}\r)}
\|f \, : \, \cMpq(\mu)\|_d
\\ &=& 
\sum_{j=1}^{\infty} 
8^d\,\exp\l\{j\l(d\log k+\l(\frac{1}{p}-\frac{1}{q}\r)\log\bt\r)\r\}\,
\|f \, : \, \cMpq(\mu)\|_d
\, \le \, C\,\|f \, : \, \cMpq(\mu)\|_d,
\end{eqnarray*}
where the constant $C$ is finite, provided 
$\ds\bt>k^{\frac{dpq}{p-q}}$. 
\end{proof}

\section{Equivalent norms of Campanato type}
\noindent

Throughout the rest of this paper 
we assume that $\mu$ satisfy the growth condition \eqref{growth}. 
We do not assume that $\mu$ is doubling. 
Before we formulate our theorems,
let us recall the definition of the RBMO spaces 
due to Tolsa \cite{To2}.
Given two cubes $Q \subset R$ with $Q \in \cQmu$, 
we denote 
\[
\dl(Q,R)
:=
\int_{\ell(Q)}^{\ell(Q_R)}
\frac{\mu(Q(z_Q,l))}{l^n}\,\frac{dl}{l},
\quad
K_{Q,R}=1+\dl(Q,R),
\]
where $Q_R$ denotes 
the smallest cube concentric to $Q$ containing $R$. 
Here and below we abbreviate 
the $(2,2^{d+1})$-doubling cube to the doubling cube 
and $\cQ(\mu;2,2^{d+1})$ to $\cQ(\mu,2)$. 
Given $Q \in \cQ(\mu)$, 
we set $Q^*$ as the smallest doubling cube 
$R$ of the form $R=2^jQ$ with 
$j=0,1,\ldots$.\footnote{
By the growth condition \eqref{growth} 
there are a lot of big doubling cubes. 
Precisely speaking, 
given a cube $Q \in \cQ(\mu)$, 
we can find $j \in \N$ with $2^jQ \in \cQ(\mu,2)$
(see \cite{To2}). 
}

Tolsa defined a new BMO for the growth measures,
which is suitable for the Calder\'{o}n-Zygmund theory. 
We say that $f \in L^1_{loc}(\mu)$ 
is an element of RBMO if it satisfies 
\[
\|f \|_*:=
\sup_{Q \in \cQ(\mu)}
\frac{1}{\mu\l(\frac{3}{2}Q\r)}
\int_Q |f(x)-m_{Q^*}(f)|\,d\mu(x)
+
\sup_{\scr Q \subset R \atop \scr Q,R \in \cQ(\mu,2)}
\frac{|m_Q(f)-m_R(f)|}{K_{Q,R}}
<\infty,
\]
where 
$\ds m_Q(f):=\frac{1}{\mu(Q)} \int_Qf(y)\,d\mu(y)$. 
Further details may be found in \cite[Section 2]{To2}.
The following lemma is due to Tolsa. 

\begin{lemma} \label{lm1}
{\rm \cite[Corollary 3.5]{To2}}
Let $f \in$RBMO. 
\begin{enumerate}
\item
There exist positive constants $C$ and  $C'$ 
independent of $f$ so that,
for every $\la>0$ and every cube $Q \in \cQmu$, 
\[
\mu\{x \in Q \, : \, 
|f(x)-m_{Q^*}(f)|>\la\}
\le C\,
\mu\l(\frac{3}{2}Q\r)\,
\exp\l(-\frac{C'\la}{\|f\|_*}\r).
\]
\item
Let $1 \le q <\infty$.
Then there exists a constant $C$ independent of $f$,
so that,
for every cube $Q \in \cQmu$,
\[
\l(\frac{1}{\mu\l(\frac{3}{2}Q\r)}
\int_Q |f(x)-m_{Q^*}(f)|^q
\,d\mu(x)\r)^{\frac{1}{q}}
\le C\,\|f\|_*.
\]
\end{enumerate}
\end{lemma}

\paragraph{Elementary property of $\dl(\cdot,\cdot)$}\ \ 
Below we list elementary properties of $\dl(\cdot,\cdot)$ 
used in this paper.

\begin{lemma} \label{lm2}
Let $Q \in \cQmu$.
Then the following properties hold {\rm :}
\begin{itemize} \item[{\rm(1)}] 
For $\rho>1$, 
we have 
$\ds
\dl(Q,\rho Q) \le C_0\,\log\rho. 
$
\item[{\rm(2)}]  
$\ds
\dl(Q,Q^*) \le C_0\,2^{n+1}\,\log2.
$
\item[{\rm(3)}]
Let $k_0 \in \N$ and $\al>0$. 
Assume, for some $\theta>0$, 
$\ds
\al \le \mu(Q) \le \mu(2^{k_0}Q) \le \theta\,\al.
$
Then 
$\ds
\dl(Q,2^{k_0}Q) \le 2^n\,\log2\cdot\theta\,C_0\,c_n,
$
where 
$\ds 
c_n:=\sum_{k=0}^{\infty}2^{-nk}.
$
\item[{\rm(4)}]
Given the cubes $P \subset Q \subset R$ with $P \in \cQmu$, 
then 
\[
\l|\dl(P,R)-\l(\dl(P,Q)+\dl(Q,R)\r)\r| \le C,
\]
where $C$ is a constant depending only on 
$C_0,n,d$. 
\item[{\rm(5)}]
Let $Q,R \in \cQmu$.
Suppose, for some constant $c_1>1$, 
$Q \subset R$ and $\ell(R) \le c_1\,\ell(Q)$. 
Then there exists a doubling cube $S \in \cQ(\mu,2)$ such that 
$Q^*,R^* \subset S$ and 
$\dl(Q^*,S),\dl(R^*,S) \le C$, 
where $C$ is a constant depending only on 
$c_1,C_0,n,d$. 
\end{itemize}
\end{lemma}

\begin{proof} 
In \cite{SaTa9}, we have proved (1)--(4). 
For reader's convenience the full proof is given here. 
(1) is obvious. To prove (2) we set $Q^*=2^{k_0}Q_0$.
We may assume that $k_0 \ge 1$.
The dyadic argument yields that 
$\ds
\dl(Q,2^{k_0}Q)
=
\int_{\ell(Q)}^{\ell(2^{k_0}Q)}
\frac{\mu(Q(z_Q,l))}{l^n}\,\frac{dl}{l}
\le
2^n\,\log2\,\sum_{k=1}^{k_0}
\frac{\mu(2^kQ)}{\ell(2^kQ)^n}.
$
Note that
$\ds
2^{d+1}\,\mu(2^{k-1}Q) \le \mu(2^kQ)
$ 
for $k=1,2,\ldots,k_0$, since $2^{k-1}Q$ is not doubling,
which yields, together with the fact that $d \ge n$,
\[
\dl(Q,2^{k_0}Q) 
\le 2^n\,\log2\,
\frac{\mu(2^{k_0}Q)}{\ell(2^{k_0}Q)^n}
\sum_{k=1}^{k_0}
(2^{n-d-1})^{k_0-k}
\le C_0\,2^{n+1}\,\log2.
\]

We prove (3). 
It follows by the dyadic argument and the assumption that 
\[
\dl(Q,2^{k_0}Q) \le 
2^n\,\log2\,\sum_{k=1}^{k_0}
\frac{\mu(2^kQ)}{\ell(2^kQ)^n}
\le 2^n\,\log2\cdot
\frac{\theta\al}{\ell(Q)^n}
\sum_{k=0}^{k_0}2^{-nk}
\le 2^n\,\log2\cdot\theta\,C_0\,c_n.
\]

Now we prove (4). It suffices to prove that 
\begin{equation} \label{b1}
A:=
\l|\dl(P_Q,R)-\dl(Q,R)\r| \le C.
\end{equation}
We decompose $A$ as 
\begin{eqnarray*}
\lft{
A \ = \ \l|
\int_{\ell(P_Q)}^{\ell(P_R)}
\frac{\mu(Q(z_P,l))}{l^n}\,\frac{dl}{l}
-\int_{\ell(Q)}^{\ell(Q_R)}
\frac{\mu(Q(z_Q,l))}{l^n}\,\frac{dl}{l}
\r|} \\ &\le& 
\int_{\ell(Q)}^{\ell(P_Q)}
\frac{\mu(Q(z_Q,l))}{l^n}\,\frac{dl}{l}
+ \l|
\int
_{\ell(P_Q)}
^{\min\{\ell(P_R),\ell(Q_R)\}}
\l(\mu(Q(z_P,l))-\mu(Q(z_Q,l))\r)
\,\frac{dl}{l^{n+1}}
\r| \\ &\quad& + \int
_{\min\{\ell(P_R),\ell(Q_R)\}}
^{\max\{\ell(P_R),\ell(Q_R)\}}
\l(\frac{\mu(Q(z_P,l))}{l^n}+\frac{\mu(Q(z_Q,l))}{l^n}\r)
\,\frac{dl}{l}
\ =: \ 
A_1+A_2+A_3.
\end{eqnarray*}
By (1) 
the integrals $A_1$ and $A_3$ are easily estimated above 
by some constant $C$. 
So we estimate $A_2$. 
Bound $A_2$ from above by 
\begin{eqnarray*}
A_2 \le \int_{\ell(P_Q)}^{\infty}
\mu(Q(z_P,l) \Delta Q(z_Q,l))
\,\frac{dl}{l^{n+1}}
= \int_{\ell(P_Q)}^{\infty}
\int_{\R^d}
\chi_{Q(z_P,l) \Delta Q(z_Q,l)}(y)
\,d\mu(y)\,\frac{dl}{l^{n+1}}.
\end{eqnarray*}
A simple geometric observation tells us that 
$\chi_{Q(z_P,l) \Delta Q(z_Q,l)}(y)=0$
if 
\[ 
l \notin 
\l[\min\{|y-z_P|_{\infty},|y-z_Q|_{\infty}\},
\,\max\{|y-z_P|_{\infty},|y-z_Q|_{\infty}\}\r],
\] 
where 
$|y|_{\infty}:=\max\{|y_1|,\ldots,|y_d|\}$. 
This observation and Fubini's theorem yield 
\begin{eqnarray*}
A_2 &\le& C\,
\int_{\R^d \setminus P_Q} \l|
\frac{1}{|y-z_P|_{\infty}{}^n}-\frac{1}{|y-z_Q|_{\infty}{}^n}
\r|\,d\mu(y)
\\ &\le& C\,
\int_{|y-z_P|_{\infty} \ge \ell(P_Q)/2}
\frac{|z_P-z_Q|_{\infty}}{|y-z_P|_{\infty}{}^{n+1}}
\,d\mu(y) \le C\,
\frac{|z_P-z_Q|_{\infty}}{\ell(P_Q)} \le C.
\end{eqnarray*}
This proves \eqref{b1}. 

Finally we establish (4).
Let $Q^*=2^jQ$. Then 
we claim $\dl(R,2^jR) \le C$. 
Indeed, 
by virtue of the fact that $Q \subset R$ we see that 
if $l \ge \ell(R)$ then 
$Q(z_R,l) \subset Q(z_q,2l)$. 
As a consequence we obtain
\begin{eqnarray*}
\lft{
\dl(R,2^jR)
=\int_{\ell(R)}^{2^j\ell(R)}
\frac{\mu(Q(z_R,l))}{l^n}\,\frac{dl}{l}
}\\
&\le&
\int_{\ell(R)}^{2^j\ell(R)}
\frac{\mu(Q(z_Q,2l))}{l^n}\,\frac{dl}{l}
\le 
\int_{\ell(Q)}^{c_1\,2^{j+1}\ell(Q)}
\frac{\mu(Q(z_Q,l))}{l^n}\,\frac{dl}{l}
\le C.
\end{eqnarray*}
If we put $S:=(2^{j+1}R)^*$, then 
$\dl(R^*,S) \le C$. 
(1) and (4) finally give us
\[
\dl(Q^*,S) \le 
\dl(Q^*,2^{j+1}R)+\dl(2^{j+1}R,S) + C 
\le C. 
\]
This is the desired result.
\end{proof}

\paragraph{Scalar-valued Campanato space}\ \ 
Having cleared up the definition of RBMO,
we will find a relationship between RBMO and the Morrey spaces. 
With the definition of RBMO in mind, 
we shall define the Campanato spaces.

Let $f \in L^1_{loc}(\mu)$.
We define the Campanato spaces $\cCpq(k,\mu)$ 
normed by 
\begin{eqnarray*}
\|f \, : \, \cCpq(k,\mu)\|
&:=&
\sup_{Q \in \cQ(\mu)}
\mu(kQ)^{\frac{1}{p}-\frac{1}{q}}
\l(\int_Q|f(x)-m_{Q^*}(f)|^q
\,d\mu(x)\r)^{\frac{1}{q}}\\
&\quad +&
\sup_{\scr Q \subset R \atop \scr Q,R \in \cQ(\mu,2)}
\mu(Q)^{\frac{1}{p}}
\frac{|m_Q(f)-m_R(f)|}{K_{Q,R}},
\,1 \le q \le p \le \infty,
\,k>1.
\end{eqnarray*}

Let $k_1,k_2>1$. 
Then $\cCpq(k_1,\mu)$ and $\cCpq(k_2,\mu)$ 
coincide as a set and their norms are mutually equivalent.
Speaking more precisely, we have the norm equivalence
\begin{equation} \label{b2}
\| f \, : \, \cCpq(k_1,\mu)\|
\sim
\| f \, : \, \cCpq(k_2,\mu)\|.
\end{equation}
To prove \eqref{b2} we may assume that $k_2=2k_1-1$
because of the monotonicity of $\cCpq(k,\mu)$
with respect to $k$.
Then all we have to prove is 
\[
\mu(k_1Q)^{\frac{1}{p}-\frac{1}{q}}
\l(\int_Q
|f(x)-m_{Q^*}(f)|^q
\,d\mu(x)\r)^{\frac{1}{q}}
\le C\,
\| f \, : \, \cCpq(k_2,\mu)\|
\]
for fixed cube $Q \in \cQmu$.
Divide equally $Q$ into $2^d$ cubes 
and collect those in $\cQmu$.
Let us name them $Q_1,Q_2,\ldots,Q_N,\,N \le 2^d$.
The triangle inequality reduces the matter to showing
\[
\mu(k_1Q)^{\frac{1}{p}-\frac{1}{q}}
\l(\int_{Q_l}|f(x)-m_{Q^*}(f)|^q\,d\mu(x)\r)^{\frac{1}{q}}
\le C\,
\| f \, : \, \cCpq(k_2,\mu)\|,
\, 1 \le l \le N.
\]
Note that $k_2Q_l \subset k_1Q$.
We apply Lemma \ref{lm2} (5) 
to obtain an auxiliary doubling cube $R$ 
which contains $(Q_l)^*,Q^*$ 
and satisfies 
$K_{(Q_l)^*,R},K_{Q^*,R} \le C$.
Thus, we obtain
\begin{eqnarray*}
\lft{
\mu(k_1Q)^{\frac{1}{p}-\frac{1}{q}}
\l(\int_{Q_l}|f(x)-m_{Q^*}(f)|^q\,d\mu(x)\r)^{\frac{1}{q}}
} \\ &\le&
\mu(k_1Q)^{\frac{1}{p}-\frac{1}{q}}
\l(\int_{Q_l}|f(x)-m_{(Q_l)^*}(f)|^q\,d\mu(x)\r)^{\frac{1}{q}}
\\ &\quad + &
\mu(Q_l)^{\frac{1}{p}}|m_{(Q_l)^*}(f)-m_R(f)|
\ + \
\mu(Q_l)^{\frac{1}{p}}|m_R(f)-m_{Q^*}(f)|
\\ &\le& C\,
\| f \, : \, \cCpq(k_2,\mu)\|.
\end{eqnarray*}
As a result \eqref{b2} is proved.

Since $\cCpq(k_1,\mu)$ and $\cCpq(k_2,\mu)$ are
isomorphic to each other as Banach spaces,
no confusion can occur if we denote 
$\cCpq(\mu)=\cCpq(2,\mu)$.

Note that $\cC^{\infty}_q(\mu)=RBMO$,
if $1\le q<\infty$.
This is an immediate consequence of Lemma \ref{lm1}.
Thus we can say 
that RBMO is a limit function space of $\cCpq(\mu)$ 
as $p \to \infty$ with $q \in [1,\infty)$ fixed.

Next, we observe $\cQ(\mu,2)$ can be seen as a net
whose order is induced by natural inclusion.
With the aid of the following proposition, 
we shall cope with the ambiguity of constant functions 
in the semi-norm of the Campanato spaces. 

\begin{proposition} \label{prp1}
Let $1\le q \le p<\infty$.
Then the limit 
$\ds M(f):=\lim_{Q \in \cQ(\mu,2)}m_Q(f)$ 
exists for every $f \in \cCpq(\mu)$. 
That is, given $\ve>0$, 
we can find a doubling cube $Q \in \cQ(\mu,2)$
such that
\[
| m_R(f)-m_Q(f) | \le \ve
\]
for all $R \in \cQ(\mu,2)$ engulfing $Q$.
In particular
there exists an increasing sequence of concentric doubling cubes 
$I_0 \subset I_1 \subset \ldots \subset I_k \subset \ldots$
such that 
\begin{equation} \label{b3}
\{m_{I_k}(f)\}_{k \in \N_0}
\mbox{ is Cauchy and }
\bigcup_kI_k=\R^d.
\end{equation} 
\end{proposition}

We remark that the condition like \eqref{b3} appears in \cite{Fu}. 
We are mainly interested in the function $f \in \cCpq(\mu)$ 
such that $M(f)=0$. 

\begin{proof}
Before we come to the proof of Proposition \ref{prp1},
we note that
\begin{eqnarray} \label{b4}
\|\,|f| \, : \, \cC^p_1(\mu)\| 
\le C\,\| f \, : \, \cC^p_1(\mu)\|.
\end{eqnarray}
Indeed, we have
\begin{eqnarray*}
\lft{
\mu\l(\frac{3}{2}Q\r)^{\frac{1}{p}-1}
\int_Q||f(x)|-m_{Q^*}(|f|)|\,d\mu(x)
} \\ &=&
\mu\l(\frac{3}{2}Q\r)^{\frac{1}{p}-1}\frac{1}{\mu(Q^*)}
\int_Q\l|\int_{Q^*}|f(x)|-|f(y)|\,d\mu(y)\r|\,d\mu(x)
\\ &\le& 
\mu\l(\frac{3}{2}Q\r)^{\frac{1}{p}-1}\frac{1}{\mu(Q^*)}
\int_Q int_{Q^*}\l||f(x)|-|f(y)|\r|\,d\mu(y)\,d\mu(x)
\\ &\le& 
\mu\l(\frac{3}{2}Q\r)^{\frac{1}{p}-1}\frac{1}{\mu(Q^*)}
\int_Q \int_{Q^*}|f(x)-f(y)|\,d\mu(y)\,d\mu(x)
\\ &\le& 
\mu\l(\frac{3}{2}Q\r)^{\frac{1}{p}-1}
\int_Q|f(x)-m_{Q^*}(f)|\,d\mu(x)
+\mu(Q^*)^{\frac{1}{p}-1}
\int_{Q^*}|m_{Q^*}(f)-f(y)|\,d\mu(y)
\\ 
&\le& 
C\,\| f \, : \, \cC^p_1(\mu)\|.
\end{eqnarray*}
In the same way we can prove
\[
\sup_{\scr Q \subset R \atop \scr Q,R \in \cQ(\mu,2)}
\mu(Q)^{\frac{1}{p}}\frac{|m_Q(|f|)-m_R(|f|)|}{K_{Q,R}}
\le C\,\| f \, : \, \cC^p_1(\mu)\|.
\]
As a consequence \eqref{b4} is justified.

We now turn to the proof of Proposition \ref{prp1}.
By the monotonicity of $\cCpq(\mu)$
with respect to $q$, we may assume $q=1$.

{\bf Case 1 $\mu$ is infinite.}
Take a sequence of concentric doubling cubes 
$\{Q_j\}_{j \in \N}$ such that for all $j \in \N$ 
\[
\mu(Q_1) \ge 1,\,\mu(Q_{j+1}) \ge 2\mu(Q_j),
\quad
\dl(Q_j,Q_{j+1}) \le C
\]
for some $C>0$ depending only on $C_0$. 
Then by the definition of $\cC^p_1(\mu)$ it holds that
\[
|m_{Q_j}(f)-m_{Q_{j+1}}(f)|
\le C\,2^{-\frac{j}{p}}
\| f \, : \, \cC^p_1(\mu)\|,
\,j \in \N.
\]
Thus we establish at least
the existence of
$\ds M(f):=\lim_{j \to \infty}m_{Q_j}(f)$.
Let $Q \in \cQmu(\mu,2)$ 
which contains $Q_j$ and 
does not contain $Q_{j+1}$. 
Set $Q'=(Q_j{}^Q)^*$. 
Then by using Lemma \ref{lm2} 
it is easy to see that 
$\dl(Q,Q') \le C$ for some absolute constant $C>0$.
Then we have
\[
|m_{Q'}(f)-m_Q(f)|,|m_{Q'}(f)-m_{Q_j}(f)|
\le C, \quad
2^{-\frac{j}{p}}\| f \, : \, \cC^p_1(\mu) \|,
\]
which implies
\[
|m_Q(f)-M(f)| 
\le C\,2^{-\frac{j}{p}}\| f \, : \, \cC^p_1(\mu) \|.
\]
Thus we finally establish 
$\ds M(f)=\lim_{Q \in \cQ(\mu,2)}m_Q(f)$.

{\bf Case 2 $\mu$ is finite.}
In this case, we have only to prove
\begin{claim} \label{cl1}
If $\mu$ is finite and $\| f \, : \, \cC^p_1(\mu) \|<\infty$, 
then $f \in L^1(\mu)$.
\end{claim}

In proving Claim \ref{cl1},
\eqref{b4} allows us to assume $f$ is positive.

We take an increasing sequence of concentric doubling cubes 
$\{Q_j\}_{j \in \N}$ 
such that
$\ds
\dl(Q_1,Q_k) \le C
$
for all $k \in \N$. 
Then we have
\[
m_{Q_k}(f) \le m_{Q_1}(f)
+\mu(Q_1)^{-\frac{1}{p}}(1+C)\| f \, : \, \cC^p_q(\mu)\|.
\]
Passage to the limit then gives
\[
\int_{\R^d}f\,d\mu \le \mu(\R^d)
\l(
m_{Q_1}(f)
+\mu(Q_1)^{-\frac{1}{p}}(1+C)\| f \, : \, \cC^p_q(\mu)\|
\r).
\]
This establishes $f \in L^1(\mu)$.
\end{proof}

The main theorem in this section is the following.

\begin{theorem} \label{thm2}
Let $1 \le q \le p<\infty$. 
Assume $f \in \cCpq(\mu)$ satisfies
$\ds
M(f)=\lim_{Q \in \cQ(\mu,2)}m_Q(f)=0.
$
Then 
\[
C^{-1}\,\| f \, : \, \cCpq(\mu)\|
\le \| f \, : \, \cMpq(\mu) \| \le
C\,\| f \, : \, \cCpq(\mu)\|
\]
for some constant $C>0$.
\end{theorem}

The left inequality is obvious. 
To prove the right inequality we need a lemma. 

\begin{lemma} \label{lm3}
Under the assumption of Theorem \ref{thm2},
given $R \in \cQ(\mu,2)$,
there exists a sequence of increasing doubling cubes 
$\{R_k\}_{k=1}^K$ such that
\begin{enumerate} \item
$R_k$ is concentric and $R_1=R$.
\item
If $\mu$ is finite, then so is $K$ 
and $R_K=\R^d$.
\item
For large $K_0 \in \N$, 
there exists $R_{k_0}$ so that 
$R_{k_0} \subset I_{K_0} \subset R_{k_0+1}$.
\item
$\mu(R_k) \ge 2^{k-1}\mu(R),\,k<K$.
\item
$\dl(R_k,R_{k+1}) \le C,\, k<K$.
\end{enumerate}
\end{lemma}

\begin{proof}
Suppose we have defined $R_k$. 
If $\mu(\R^d) \le 2^k\mu(R)$,
then we set $R_{k+1}=\R^d$ and we stop.
Suppose otherwise.
We define $R_{k+1}$ as the smallest doubling cube
of the form $2^lR_k$ with $l \ge 3$
whose $\mu$-measure exceeds $2^k\mu(R)$.
By virtue of Lemma \ref{lm2} (3) 
it is easy to verify 
that $\{R_k\}_{k=1}^K$ obtained in this way 
satisfies the property of the lemma. 
\end{proof}

Let us return to the proof of Theorem \ref{thm2}.
Let $R \in \cQ(\mu)$. 
We shall estimate 
\[
\mu(2R)^{\frac{1}{p}-\frac{1}{q}}
\l(\int_R|f(x)|^q\,d\mu(x)\r)^{\frac{1}{q}}.
\]
The triangle inequality enables us 
to majorize the above integral by 
\[
\mu\l(\frac{3}{2}R\r)^{\frac{1}{p}-\frac{1}{q}}
\l(\int_R |f(x)-m_{R^*}(f)|^q\,d\mu(x)\r)^{\frac{1}{q}}
+ \mu(R)^\frac{1}{p}|m_{R^*}(f)|.
\]
Consequently 
we can reduce the matters to the estimate of
$\ds
\mu(R^*)^{\frac{1}{p}}|m_{R^*}(f)|.
$

Now we invoke Lemma \ref{lm3}
for $K_0$ taken so that
$\ds
\mu(R)^{\frac{1}{p}}|m_{I_{K_0}}(f)|
\le \| f \, : \, \cCpq(\mu)\|.
$
Using the sequence $\{R_k\}_{k=1}^K$,
we obtain
\begin{eqnarray*}
\lft{
\mu(R^*)^\frac{1}{p}\,|m_{R_k}(f)-m_{R_{k+1}}(f)|
} \\ &\le& C\,
2^{-\frac{k}{p}}\,
\mu(R_k)^\frac{1}{p}\,
\frac{|m_{R_k}(f)-m_{R_{k+1}}(f)|}{1+\dl(R_k,R_{k+1})}
\,\le\, C\,
2^{-\frac{k}{p}}
\|f_j \, : \, \cCpq(l^r,\mu)\|.
\end{eqnarray*}
We also have 
$\ds
\mu(R^*)^{\frac{1}{p}}
|m_{R_{k_0}}(f)-m_{I_{K_0}}(f)|
\le C\,2^{-\frac{k_0}{p}}\|f \, : \, \cCpq(\mu)\|,
$
since by the properties {\it 3} and {\it 4} of Lemma \ref{lm3}
we see that
$\dl(R_{k_0},R_{k_0+1}),\dl(I_{K_0},R_{k_0+1})$
are majorized by some constants dependent only on $C_0$. 
The triangle inequality gives us
\begin{eqnarray*}
\lft{
\mu(R^*)^\frac{1}{p}\,|m_{R^*}(f)|
} \\ &\le& 
\mu(R)^\frac{1}{p}\,
\sum_{k=1}^{k_0-1}|m_{R_k}(f)-m_{R_{k+1}}(f)|
+
\mu(R^*)^\frac{1}{p}\,
\l(
|m_{R_{k_0}}(f)-m_{I_{K_0}}(f)|
+ 
|m_{I_{K_0}}(f)|
\r)
\\ &\le& C\,
\l(
\sum_{k=1}^{\infty}
2^{-\frac{k}{p}}
\r)\,
\| f \, : \, \cCpq(\mu)\|
+
\mu(R^*)^\frac{1}{p}\,|m_{I_{K_1}}(f)|
\,\le\, C\,\|f \, : \, \cCpq(\mu)\|.
\end{eqnarray*} 
The proof of Theorem \ref{thm2} is therefore complete.

\paragraph{Vector-valued extension}\ \ 
Finally we consider the vector-valued extensions of Theorem \ref{thm2}.
Let $\| a_j \, : \, l^r \|$ denote
the $l^r$-norm of $a=\{a_j\}_{j \in \N}$. 
If possible confusion can occur,
then we write 
$\|\{a_j\}_{j \in \N} \, : \, l^r\|$.
For $f \in L^1_{loc}(\mu)$, 
we define the sharp maximal operator due to Tolsa by 
\[
M^{\sharp}f(x)
:=
\sup_{x \in Q \in \cQ(\mu)}
\frac{1}{\mu\l(\frac{3}{2}Q\r)}
\int_Q|f(y)-m_{Q^*}(f)|\,d\mu(y)
+
\sup_{\scr x \in Q \subset R \atop \scr Q,R \in \cQ(\mu,2)}
\frac{|m_Q(f)-m_R(f)|}{K_{Q,R}}.
\]

Lemma \ref{lm1} can be extended 
to the following vector-valued version. 

\begin{lemma} \label{lm4}
{\rm \cite{SaTa9}}
Let $f_j \in$RBMO for $j=1,2,\ldots$. 
For any cube $Q \in \cQ(\mu)$ and 
$q,r \in (1,\infty)$, 
there exists a constant $C$ independent of $f_j$ such that 
\begin{equation} \label{b5}
\l(
\frac{1}{\mu\l(\frac{3}{2}Q\r)}
\int_Q
\|f_j(x)-m_{Q^*}(f_j) \, : \, l^r\|^q
\,d\mu(x)\r)^{\frac{1}{q}}
\le C\,\sup_{x \in \R^d}
\l\|M^{\sharp}f_j(x) \, : \, l^r\r\|.
\end{equation}
\end{lemma}

We now define the vector-valued Campanato spaces. 
Let $1 \le q \le p \le \infty$ and $r \in (1,\infty)$. 
We say that $\{f_j\}_{j \in \N}$ belongs to 
the vector-valued Campanato spaces $\cCpq(l^r,\mu)$ if 
each $f_j$ is $\mu$-measurable and 
\begin{eqnarray*}
\|f_j \, : \, \cCpq(l^r,\mu)\|
&:=&
\sup_{Q \in \cQ(\mu)}
\mu(2Q)^{\frac{1}{p}-\frac{1}{q}}
\l(\int_Q
\|f_j(x)-m_{Q^*}(f_j) \, : \, l^r\|^q
\,d\mu(x)\r)^{\frac{1}{q}}
\\ &\quad +&
\sup_{\scr Q \subset R \atop \scr Q,R \in \cQ(\mu,2)}
\mu(Q)^{\frac{1}{p}}
\frac{\|m_Q(f_j)-m_R(f_j) \, : \, l^r\|}{K_{Q,R}}
<\infty.
\end{eqnarray*}

As for the vector-valued spaces, 
the norm equivalence of the Campanato type still holds. 

\begin{theorem} \label{thm3}
Let $1 \le q \le p<\infty$ and 
let $\{f_j\}_{j \in \N}$ be a sequence 
in $\cCpq(\mu)$. 
Assume that there exists an increasing sequence 
of concentric doubling cubes 
$I_0 \subset I_1 \subset \ldots \subset I_k \subset \ldots$
such that 
\[
\lim_{k \to \infty}m_{I_k}(f_j)=0
\mbox{ for all $j$ and }
\bigcup_kI_k=\R^d.
\]
Then there exists a constant $C>0$ 
independent of $\{f_j\}_{j \in \N}$ such that 
\[
C^{-1}\,\|f_j \, : \, \cCpq(l^r,\mu)\|
\le \|f_j \, : \, \cMpq(l^r,\mu)\| \le
C\,\|f_j \, : \, \cCpq(l^r,\mu)\|.
\]
\end{theorem}

Using Lemma \ref{lm4}, 
we can say more about $\cC^{\infty}_q(l^r,\mu)$, 
which gives us a partial clue to the definition 
of the vector-valued RBMO spaces.
Speaking precisely,
we obtain the following proposition.

\begin{proposition} \label{prp3}
Let $\{f_j\}_{j \in \N}$ be a sequence 
of $L^1_{loc}(\mu)$ functions.
Then
\begin{equation} \label{b6}
\sup_{\scr Q \subset R \atop \scr Q,R \in \cQ(\mu,2)}
\frac{\|m_Q(f_j)-m_R(f_j) \, : \, l^r\|}{K_{Q,R}} 
\le c\,
\sup_{x \in \R^d}\|M^{\sharp}f_j(x) \, : \, l^r\|.
\end{equation}
In particular, we have
\begin{equation} \label{b7}
\|f_j \, : \, \cC^{\infty}_q(l^r,\mu) \|
\le c\,
\sup_{x \in \R^d}\| M^{\sharp}f_j(x) \, : \, l^r\|.
\end{equation}
\end{proposition}

\begin{proof}
Fix $Q \subset R$ such that $Q \in \cQmu$.
Then
$\ds
\frac{|m_Q(f_j)-m_R(f_j)|}{K_{Q,R}} \le c\,M^\sharp f_j(x)
$
for all $x \in Q$.
By taking the $l^r$-norm of both sides we obtain
\[
\frac{\|m_Q(f_j)-m_R(f_j) \, : \, l^r\|}{K_{Q,R}}
\le c\,\sup_{x \in Q}\| M^\sharp f_j(x) \, : \, l^r \|
\le c\,\sup_{x \in \R^d}\| M^\sharp f_j(x) \, : \, l^r \|.
\]
Now since $Q$ and $R$ are taken arbitrarily,
\eqref{b6} is proved.
\eqref{b7} can be obtained
with the help of \eqref{b5} and \eqref{b6}.
\end{proof}

Before we conclude this section,
a remark may be in order.
\begin{remark}
Let $0<\al<n$.
For $Q,R \in \cQmu$ with $Q \subset R$
we define
\[
K_{Q,R}^{(\al)}
=1+\sum_{k=1}^{N_{Q,R}}
\l(\frac{\mu(2^kQ)}{\ell(2^kQ)^n}\r)^{\frac{n-\al}{n}},
\]
where $N_{Q,R}$ is the least integer $j$
with $2^jQ \supset R$.
For the definition of this constant 
we refer to \cite{ChSa}.
Theorems in this section still hold, 
if we replace $K_{Q,R}$
by $K_{Q,R}^{(\al)}$
whenever $1 \le q \le p<\infty$.
\end{remark}

\section*{Acknowledgement}
\noindent

The authors thank Prof.~E.~Nakai
for discussing this paper with them.

\end{document}